\documentclass[11pt]{article}

\usepackage[a4paper,margin=1in]{geometry}
\usepackage[T1]{fontenc}
\usepackage{lmodern}
\usepackage{amsmath}
\usepackage{amsthm}
\usepackage[hidelinks]{hyperref}

\theoremstyle{plain}
\newtheorem{theorem}{Theorem}[section]
\newtheorem{lemma}[theorem]{Lemma}
\newtheorem{proposition}[theorem]{Proposition}
\newtheorem{corollary}[theorem]{Corollary}
\theoremstyle{definition}
\newtheorem{definition}[theorem]{Definition}

\newcommand{\headers}[2]{}
\newcommand{\funding}[1]{#1}
\newcommand{\email}[1]{\href{mailto:#1}{\nolinkurl{#1}}}
\newenvironment{keywords}{\par\smallskip\noindent\textbf{Keywords.}\ }{\par\smallskip}
\newenvironment{MSCcodes}{\par\noindent\textbf{MSC codes.}\ }{\par\medskip}

\numberwithin{equation}{section}
\date{}

\usepackage{amsmath,amssymb,mathtools,booktabs}
\usepackage{graphicx}
\usepackage{enumitem}
\usepackage{microtype}
\usepackage{mathrsfs}
\newtheorem{remark}[theorem]{Remark}

\newcommand{\R}{\mathbb R}
\newcommand{\C}{\mathbb C}
\newcommand{\Hpop}{\mathcal H}
\newcommand{\Wmac}{\mathcal W}
\newcommand{\Klift}{\mathcal K}
\newcommand{\Smom}{\mathcal S}
\newcommand{\Jrev}{\mathcal J}
\newcommand{\Trans}{\mathcal T}
\newcommand{\Coll}{\mathcal C}

\newcommand{\Ker}{\operatorname{ker}}
\newcommand{\diag}{\operatorname{diag}}
\newcommand{\norm}[1]{\left\lVert#1\right\rVert}
\newcommand{\ip}[2]{\left\langle#1,#2\right\rangle}
\newcommand{\adj}{\dagger}

\headers{MESH-UNIFORM STABILITY WITH REVERSIBLE BOUNDARIES}{J. ZHAO}

\title{Mesh-Uniform Power Stability of Two-Relaxation-Time Vector Lattice Boltzmann Schemes with Reversible Boundaries\thanks{\funding{This work was supported by the Beijing Natural Science Foundation (No. JR25003) and the National Natural Science Foundation of China (grants 12301520 and 12671504).}}}

\author{Jin Zhao\thanks{Academy for Multidisciplinary Studies, Capital Normal University, Beijing 100048, China (\email{zjin@cnu.edu.cn}).}}

\hypersetup{
  pdftitle={Mesh-Uniform Power Stability of Two-Relaxation-Time Vector Lattice Boltzmann Schemes with Reversible Boundaries},
  pdfauthor={Jin Zhao},
  pdfsubject={Mesh-uniform power stability of vector two-relaxation-time lattice Boltzmann schemes with reversible boundaries},
  pdfkeywords={vector lattice Boltzmann method, two-relaxation-time scheme, initial-boundary value problems, mesh-uniform power stability, numerical range, nonnormal operators}
}

\begin{document}
\maketitle

\begin{abstract}
We consider the collision--transport operator for vector-valued two-relaxation-time lattice Boltzmann schemes linearized about a uniform rest state.  Assume that the equilibrium blocks are positive definite and that the link-even and link-odd relaxation parameters satisfy $s_++s_-=2$ and $0<s_-<2$.  If the homogeneous transport is unitary in the equilibrium metric and reversible under velocity exchange, then the powers of the amplification operator are bounded uniformly with respect to the number and arrangement of lattice nodes.  The admissible transports include periodic transport, vector halfway bounce-back, coordinate-aligned specular reflection, tangential orthogonal involutions, and compatible multi-channel scattering.  The proof reduces the population equation to a two-step macroscopic recurrence generated by a contraction.  An inclusion of the numerical range in an ellipse, combined with the Crouzeix--Palencia theorem, yields a dimension-independent estimate for the companion operator and hence the population bound.  For a three-coefficient off-midpoint boundary interpolation, we give an exact rational D2N5 example whose finite-domain amplification matrix has a real eigenvalue larger than one, although the interpolation coefficients and bulk parameters are admissible.  Thus coefficient convexity alone does not ensure stability for this boundary family.
\end{abstract}



\section{Introduction}
\label{sec:introduction}

Lattice Boltzmann methods alternate local relaxation of kinetic populations with exact translation along a finite velocity set.  Classical kinetic and hydrodynamic formulations appear in \cite{BenziSucciVergassola1992,QianDHumieresLallemand1992,HeLuo1997}.  Finite-difference and asymptotic interpretations \cite{Junk2001,JunkYong2003,JunkKlarLuo2005,BellottiGrailleMassot2022} and linear or weighted-energy stability analyses \cite{SterlingChen1996,LallemandLuo2000,BandaYongKlar2006,JunkYong2009,Rheinlander2010} provide the bulk framework used here.

Boundary closure requires a separate stability analysis, because the boundary operator may introduce unstable modes or substantial nonnormal growth even when the periodic interior scheme is stable.  Power and resolvent estimates are classical in matrix stability \cite{Kreiss1959}, and the distinction between the Cauchy problem and the boundary-closed problem is fundamental in the theory of hyperbolic initial-boundary value problems and difference approximations \cite{Kreiss1970,GustafssonKreissSundstrom1972,GustafssonKreissOliger1995}.  Bounded-domain lattice Boltzmann analyses appear in \cite{JunkYangBoundary2005,JunkYang2009}.  Recent scalar work treats strong stability and boundary symbols in population variables \cite{Bellotti2026Raw}; a two-velocity study analyzes spectral and pseudospectral behavior \cite{Bellotti2026TwoVelocities}.

Vector populations arise in discrete-kinetic and relaxation representations of conservation laws \cite{Bouchut1999,AregbaDriolletNatalini2000,Graille2014} and in kinetic models for magnetization and magnetohydrodynamics \cite{GuyerMcCall2000,Dellar2002,BatyEtAl2023}.  Related analyses address over-relaxation \cite{DruiEtAl2019}, entropy-stable relaxation \cite{BellottiHelluyNavoret2025,GuillonHelieHelluy2024}, and positivity or bound preservation \cite{WissocqLiuAbgrall2025}.

Diffusive vector-BGK approximations of incompressible flow were developed in \cite{CarforaNatalini2008,ZhaoZhangYong2020JMAA,Zhao2021}; vector boundary closures and a one-rate weighted $L^2$ estimate for the homogeneous halfway rule appear in \cite{ZhaoZhangYong2020}.  Off-lattice wall treatments include interpolation and multireflection closures \cite{MeiLuoShyy1999,BouzidiFirdaoussLallemand2001,GinzburgDHumieres2003}, single-node second-order constructions \cite{ZhaoYong2017,ZhaoHuangYong2019,MarsonEtAl2021}, a parametric vectorial finite-difference family \cite{ZhangFengZhao2021}, and the unified directional LI$^+$/EMR framework \cite{GinzburgEtAl2023Unified}.  The present work addresses the $\ell^2$ power stability of the rest-state linearized collision--transport operator on bounded lattices.

For each opposite-velocity pair, a two-relaxation-time (TRT) collision assigns distinct rates to the link-even and link-odd nonequilibrium components.  Scalar TRT work treats equilibrium and link formulations \cite{Ginzburg2005}, hydrodynamic solutions \cite{GinzburgSimple2008}, boundary parametrizations \cite{GinzburgTRT2008}, and truncation error and stability \cite{Ginzburg2012}.  We use
\begin{equation}
 s_++s_-=2,
 \label{eq:otrt-intro}
\end{equation}
often called the optimal TRT (OTRT) relation; the name is only terminology here.  In an equivalent forward/backward formulation, Dellar \cite{Dellar2024} observed that \eqref{eq:otrt-intro} sets the forward coefficient to one and permits a two-step macroscopic recurrence.  Dellar's later two-rate projection \cite{Dellar2025}, acting on symmetric and antisymmetric parts of the momentum flux, is distinct from the link-parity collision below.

Recent results address related but different questions.  For scalar nonlinear TRT schemes on the whole lattice, Aregba-Driollet and Bellotti \cite{AregbaDriolletBellotti2026TRT} establish componentwise monotonicity, $L^1$ contractivity, and convergence under grid refinement to the entropy solution.  Their framework contains the same rate-sum relation used here, called the magic combination there, and recalls the corresponding scalar two-step formulation; it does not analyze physical boundary scattering.  Their equilibrium boundary construction for multidimensional vectorial schemes \cite{AregbaDriolletBellotti2026Boundary} instead fills missing incoming populations with equilibrium values computed from exterior macroscopic states.  Its convergence theorem is for the scalar monotone case; systems are treated numerically.  A recent vectorial incompressible Navier--Stokes construction \cite{AregbaDriolletBellottiNataliniTenna2026} establishes second-order consistency, studies linearized spectral behavior, and reports nonlinear numerical tests.  The present paper studies homogeneous reversible scattering of the complete population state, proves a mesh-uniform bound for its vector rest-state linearization, and gives an exact instability result for a different off-midpoint boundary discretization.

A local collision estimate does not control the boundary-closed product: the natural equilibrium--kinetic quadratic form has cross-velocity blocks and is generally not invariant under streaming.  Factorwise contractivity in one microscopic norm is sufficient, not necessary.  In the equilibrium metric, however, the equilibrium lifting is isometric and homogeneous reversible transport is unitary; hence their macroscopic compression $P$ is a contraction.  Here $w^n=(w^n(x))_{x\in\Lambda}$ collects the nodal density--momentum perturbations, with $w^n(x)=(\delta\rho^n(x),\delta q^n(x)^T)^T$ obtained by summing the population perturbations at $x$.  On \eqref{eq:otrt-intro}, the complete linearized population update reduces exactly to
\[
 w^{n+1}=(P+bP^*)w^n-bw^{n-1},\qquad |b|<1.
\]

A numerical-range argument gives a dimension-independent companion estimate.  The numerical range of $P+bP^*$ lies in an ellipse determined by $b$, where the scalar recurrence polynomials are uniformly bounded; the Crouzeix--Palencia theorem \cite{CrouzeixPalencia2017} transfers this bound to the nonnormal operator.  Lifting the estimate and bounding velocity reversal give a population-space constant depending only on $a$, $\alpha$, and $s_-$, uniformly over finite lattices and reversible transports.  These include periodic transport, homogeneous vector halfway bounce-back, coordinate-aligned specular reflection, tangential orthogonal involutions, and compatible patchwise mixing.

The analysis also identifies an obstruction to a direct extension of the one-rate factorwise energy argument: when the lattice-speed parameter is nonzero, no block-diagonal positive-definite symmetrizer can both orthogonalize the equilibrium projection and commute with velocity reversal.  This obstruction concerns that separable proof strategy, not stability of the complete product.

For generic coefficients, the three-coefficient off-midpoint interpolation considered later mixes pre- and postcollision populations and lies outside the collision--transport factorization used here; coefficient-degenerate cases require a separate reversibility check.  We give an exact rational D2N5 construction with strictly positive interpolation coefficients and admissible bulk parameters whose full amplification matrix has a real eigenvalue larger than one.  The counterexample excludes the full convex-coefficient range as an unconditional stability region; it does not assert that every off-midpoint parameter choice is unstable.

The uniform theorem concerns the homogeneous rest-state linearization.  It gives an $\ell^2$ power bound and a uniform exterior resolvent estimate.  For a flat boundary, these supply the homogeneous stability component expected in a GKS analysis, but they do not include the full boundary-trace estimate for arbitrary incoming data.  Nonzero-background linearizations, nonlinear stability, and a completed second-order energy-compatible off-midpoint rule remain separate questions.

Section~\ref{sec:model} defines the linearization, and Section~\ref{sec:transport} the reversible transport class.  Sections~\ref{sec:companion}--\ref{sec:stability} prove the companion and population bounds.  Section~\ref{sec:interpolation} gives the exact off-midpoint instability result, and Section~\ref{sec:discussion} discusses boundary symbols and energy-compatible extensions.  The Supplementary Materials give the rational data underlying the counterexample and the finite-domain calculations; the associated programs and numerical data are supplied in the accompanying archive.

\section{The rest-state vector TRT linearization}
\label{sec:model}

This section formulates the rest-state linearized vector TRT collision and records the algebraic identities needed for the later population-to-macroscopic reduction.  It also shows why no block-diagonal symmetric positive definite metric can simultaneously orthogonalize the equilibrium projection and commute with velocity reversal.

\subsection{Discrete velocities and equilibrium Jacobians}

Let $d\in\{1,2,3\}$ and $m=d+1$.  The lattice velocity set is
\[
 \mathcal V=\{+e_1,\ldots,+e_d,-e_1,\ldots,-e_d,0\}\subset\mathbb Z^d.
\]
At a fixed node and time level, the scheme stores a vector population $F_i\in\R^m$ for each $i\in\mathcal V$: one density contribution and $d$ momentum contributions.  Suppressing node and time indices, write $W:=\sum_iF_i=(\rho,q^T)^T$ for the local density--momentum state.  The opposite velocity is $\bar i$; the rest velocity is self-opposite.  On an open set $\mathcal U\subset\R^m$ of admissible states, assume $M_i\in C^1(\mathcal U;\R^m)$ and
$\sum_{i\in\mathcal V}M_i(W)=W$ for $W\in\mathcal U$.

Before linearization, the local pairwise collision has the form
\begin{equation}
 F_i^*=F_i+\phi\bigl(M_i(W)-F_i\bigr)
          +\psi\bigl(M_{\bar i}(W)-F_{\bar i}\bigr),
 \label{eq:nonlinear-collision}
\end{equation}
with $\phi,\psi\in\R$, followed by transport and boundary scattering.  Fix a uniform rest state $W_0=(\rho_0,0^T)^T\in\mathcal U$, $\rho_0>0$, and set $f_i:=F_i-M_i(W_0)$, $f_i^*:=F_i^*-M_i(W_0)$, and $w:=W-W_0=\sum_i f_i=(\delta\rho,\delta q^T)^T$.  With $K_i:=D_WM_i(W_0)$, differentiability gives $M_i(W_0+w)-M_i(W_0)=K_iw+o(\norm{w})$.  Thus the linearized collision is
\begin{equation}
 f_i^*=f_i+\phi(K_iw-f_i)+\psi(K_{\bar i}w-f_{\bar i}).
 \label{eq:linearized-collision}
\end{equation}
Hereafter lowercase variables denote perturbations: $f=(f_i)_i$ is the full population state and $w$ its density--momentum moment.  We ask whether the collision in \eqref{eq:linearized-collision}, followed by transport with homogeneous boundary scattering, is power bounded uniformly over finite lattices and admissible boundaries.

Let $\mathbf e_0,\ldots,\mathbf e_d$ denote the canonical basis of $\R^m$, to distinguish it from the lattice velocities.  Define
\begin{equation}
 B_j=\mathbf e_0\mathbf e_j^T+\mathbf e_j\mathbf e_0^T,
 \qquad j=1,\ldots,d,
 \label{eq:Bj}
\end{equation}
and
\begin{equation}
 K_{+j}=aI_m+\frac{\alpha}{2}B_j,\qquad
 K_{-j}=aI_m-\frac{\alpha}{2}B_j,\qquad
 K_0=(1-2da)I_m.
 \label{eq:Ki}
\end{equation}
These are the rest-state Jacobians of the diffusive incompressible vector equilibria in \cite{ZhaoZhangYong2020JMAA,Zhao2021}.  Directly from \eqref{eq:Ki},
\begin{equation}
 \sum_{i\in\mathcal V}K_i=I_m.
 \label{eq:sumK}
\end{equation}
Let
\begin{equation}
 R_0=\diag(1,-I_d).
 \label{eq:R0}
\end{equation}
Then
\begin{equation}
 K_{\bar i}=R_0K_iR_0
 \label{eq:Kconjugacy}
\end{equation}
for each moving velocity.

\begin{lemma}[Positivity of the equilibrium blocks]
\label{lem:K-positive}
If
\begin{equation}
 0<\alpha,\qquad \frac{\alpha}{2}<a<\frac1{2d},
 \label{eq:positive-range}
\end{equation}
then every $K_i$ is symmetric positive definite.
\end{lemma}

\begin{proof}
On $\operatorname{span}\{\mathbf e_0,\mathbf e_j\}$, the matrix $B_j$ has eigenvalues $1$ and $-1$.  It vanishes on the orthogonal complement.  Hence the eigenvalues of $K_{\pm j}$ are $a+\alpha/2$, $a-\alpha/2$, and $a$ with multiplicity $d-1$.  The rest block has the eigenvalue $1-2da$ with multiplicity $m$.  All are positive under \eqref{eq:positive-range}.
\end{proof}

\subsection{Parity rates and global operators}

For a moving pair, set
\[
 f_i^+=\frac{f_i+f_{\bar i}}2,\qquad
 f_i^-=\frac{f_i-f_{\bar i}}2,
\]
and define $K_i^+=(K_i+K_{\bar i})/2$ and $K_i^-=(K_i-K_{\bar i})/2$.  Taking the half-sum and half-difference of the $i$ and $\bar i$ equations in \eqref{eq:linearized-collision} gives
\begin{equation}
 (f_i^+)^*=f_i^++s_+\bigl(K_i^+w-f_i^+\bigr),\qquad
 (f_i^-)^*=f_i^-+s_-\bigl(K_i^-w-f_i^-\bigr),
 \label{eq:parity-collision}
\end{equation}
where
\begin{equation}
 s_+=\phi+\psi,\qquad s_-=\phi-\psi,
 \qquad
 \phi=\frac{s_++s_-}{2},\qquad
 \psi=\frac{s_+-s_-}{2}.
 \label{eq:rates}
\end{equation}
Under the diffusive space--time scaling used for the incompressible models in \cite{ZhaoZhangYong2020JMAA,Zhao2021}, $s_-$ is the viscosity-setting link-odd rate, whereas $s_+$ is a kinetic or boundary-tuning rate.  The dimensional conversion depends on the scaling normalization and is not used in the stability argument below.

Let $\Lambda\subset\mathbb Z^d$ be nonempty and finite.  Define
\[
 \Hpop_\Lambda=(\R^m)^{\Lambda\times\mathcal V},\qquad
 \Wmac_\Lambda=(\R^m)^\Lambda.
\]
The moment map, equilibrium lifting, and velocity reversal are
\begin{equation}
 (\Smom f)(x)=\sum_i f_i(x),\qquad
 (\Klift w)_i(x)=K_iw(x),\qquad
 (\Jrev f)_i(x)=f_{\bar i}(x).
 \label{eq:SKJ}
\end{equation}
Equation \eqref{eq:sumK} implies
\begin{equation}
 \Smom\Klift=I_{\Wmac_\Lambda},\qquad
 E:=\Klift\Smom,\qquad E^2=E.
 \label{eq:Eprojection}
\end{equation}
The global linearized rest-state collision operator is
\begin{equation}
 \Coll=I+(\phi I+\psi\Jrev)(E-I).
 \label{eq:global-collision}
\end{equation}
On the rate relation
\begin{equation}
 s_++s_-=2,
 \label{eq:otrt}
\end{equation}
we have $\phi=1$.  With
\begin{equation}
 b=\psi=1-s_-,
 \label{eq:b}
\end{equation}
condition $0<s_-<2$ is exactly $|b|<1$, and \eqref{eq:global-collision} becomes
\begin{equation}
 \Coll f=(I+b\Jrev)\Klift\Smom f-b\Jrev f.
 \label{eq:collision-otrt}
\end{equation}
The one-rate condition $s_+=s_-$ intersects the OTRT relation \eqref{eq:otrt}
only at $s_+=s_-=1$; thus the analysis below concerns the admissible open
segment $0<s_-<2$ of the OTRT line, not the general one-rate family.

\subsection{Local collision dissipation and a structural obstruction}

The following identity records what local collision control does and does not provide.  Define
\[
 H_{\rm coll}:=E^TE+(I-E)^T(I-E),\qquad
 \norm{f}_{H_{\rm coll}}^2=f^TH_{\rm coll}f.
\]
Since $E$ is a projection,
\[
 \ip{u}{v}_{H_{\rm coll}}
 =\ip{Eu}{Ev}_2+\ip{(I-E)u}{(I-E)v}_2
\]
defines an inner product, so $H_{\rm coll}$ is positive definite.  Moreover,
$\Ker E=\Ker\Smom$: one inclusion follows from $E=\Klift\Smom$, and the
other from $\Smom E=\Smom$.  For $g=(I-E)f$, let
$g_+=(I+\Jrev)g/2$ and $g_-=(I-\Jrev)g/2$.  Velocity reversal preserves
$\Ker E$ because $\Smom\Jrev=\Smom$.

\begin{proposition}[Exact local collision identity]
\label{prop:local-collision}
For every $f$,
\begin{equation}
 \norm{\Coll f}_{H_{\rm coll}}^2-\norm{f}_{H_{\rm coll}}^2
 =-s_+(2-s_+)\norm{g_+}_2^2
  -s_-(2-s_-)\norm{g_-}_2^2.
 \label{eq:local-identity}
\end{equation}
Consequently, local collision is nonexpansive in $\norm{\cdot}_{H_{\rm coll}}$ if and only if $0\le s_+,s_-\le2$.
\end{proposition}

\begin{proof}
Write $f=Ef+g_++g_-$.  The range of $E$ and its kernel are orthogonal in
$\ip{\cdot}{\cdot}_{H_{\rm coll}}$.  On $\Ker E$ this inner product is Euclidean, while
$g_+$ and $g_-$ belong to the $+1$ and $-1$ eigenspaces of the Euclidean
orthogonal involution $\Jrev$; hence they are orthogonal.  Collision fixes
$Ef$ and multiplies $g_+$ and $g_-$ by $1-s_+$ and $1-s_-$, respectively.
Subtraction of the squared norms now gives \eqref{eq:local-identity}.

For necessity, choose a moving pair and a nonzero $h\in\R^m$.  The vector with $g_i=h$, $g_{\bar i}=-h$, and all other blocks zero is a nonzero odd vector in $\Ker E$.  The vector with $g_i=g_{\bar i}=h$, $g_0=-2h$, and all other blocks zero is a nonzero even vector in $\Ker E$.  Testing \eqref{eq:local-identity} on these two vectors shows that nonexpansiveness requires $s_\pm(2-s_\pm)\ge0$.
\end{proof}

The global matrix $H_{\rm coll}$ is a direct sum over nodes, but its block at each node generally contains cross-velocity terms.  Streaming sends different velocity blocks to different nodes and therefore need not preserve the associated global quadratic form.  Thus $H_{\rm coll}$ is an auxiliary collision metric, not the transport-compatible equilibrium metric used in the main stability theorem.  One might instead seek a velocity-block-diagonal metric that is compatible with streaming.  At one lattice node, set $K_{\rm loc}=\operatorname{col}(K_i)_i$, $S_{\rm loc}=(I_m\ \cdots\ I_m)$, and let $E_{\rm loc}=K_{\rm loc}S_{\rm loc}$ denote the local equilibrium projection.  The next result records an obstruction to this particular separable construction.

\begin{proposition}[Absence of a common separable orthogonalizer]
\label{prop:no-common}
Assume $\alpha>0$.  There is no block-diagonal symmetric positive definite matrix $G=\diag(G_i)$ on the local velocity space such that
\begin{equation}
 GE_{\rm loc}=E_{\rm loc}^TG,\qquad G\Jrev=\Jrev G.
 \label{eq:no-common}
\end{equation}
\end{proposition}

\begin{proof}
The $(i,j)$ block of $GE_{\rm loc}=E_{\rm loc}^TG$ is, using the symmetry
of every $K_j$,
\[
 G_iK_i=K_j^TG_j=K_jG_j.
\]
Taking $i=j$ shows that $G_i$ commutes with $K_i$.  The second equality in \eqref{eq:no-common} gives $G_i=G_{\bar i}$.  Taking $j=\bar i$ and using commutation for the opposite block yields
\[
 G_iK_i=K_{\bar i}G_i=G_iK_{\bar i}.
\]
Since $G_i$ is invertible, $K_i=K_{\bar i}$, contrary to \eqref{eq:Ki} when $\alpha>0$.
\end{proof}

Proposition~\ref{prop:no-common} rules out this particular factorwise argument, not stability of the product.

\section{Reversible boundary transport}
\label{sec:transport}

We now equip the population space with the equilibrium metric, identify the reversible transports that are unitary in that metric, and form the macroscopic compression $P=\Smom\Trans\Klift$, which is a contraction.  These properties also supply the adjoint identity used in the two-step recurrence of Section~\ref{sec:stability}.

\subsection{The equilibrium metric}

Under \eqref{eq:positive-range}, define the global equilibrium weight
\[
 D_\Lambda:=\bigoplus_{x\in\Lambda}\bigoplus_{i\in\mathcal V}K_i^{-1},
 \qquad (D_\Lambda f)_i(x)=K_i^{-1}f_i(x).
\]
It is symmetric positive definite by Lemma~\ref{lem:K-positive}.  The same
equilibrium-block weight underlies the one-rate vector boundary estimate in
\cite{ZhaoZhangYong2020}; here it is used as the equilibrium lifting and
transport metric, not as a common factorwise TRT collision symmetrizer.  When
$\Lambda$ is clear, equip $\Hpop_\Lambda$ with the shorthand
\begin{equation}
 \ip{f}{g}_D:=f^TD_\Lambda g
 =\sum_{x\in\Lambda}\sum_{i\in\mathcal V}
 f_i(x)^TK_i^{-1}g_i(x),
 \qquad \norm{f}_D^2=\ip{f}{f}_D,
 \label{eq:Dmetric}
\end{equation}
and equip $\Wmac_\Lambda$ with the Euclidean inner product.  Accordingly,
regard the equilibrium lifting as the linear map
\[
 \Klift:
 (\Wmac_\Lambda,\langle\cdot,\cdot\rangle_2)
 \longrightarrow
 (\Hpop_\Lambda,\langle\cdot,\cdot\rangle_D),
 \qquad (\Klift w)_i(x)=K_iw(x).
\]
The equilibrium projection $E$ is orthogonal in this metric, since
$D_\Lambda E=E^TD_\Lambda$.  For $\alpha>0$, however,
$D_\Lambda\Jrev\ne\Jrev D_\Lambda$; this is consistent with
Proposition~\ref{prop:no-common}.  The reversal is controlled later by a
uniform norm bound, while transport compatibility is imposed linkwise below.
The symbol $\adj$ denotes the adjoint for maps between these two inner-product
spaces; $^*$ denotes the Euclidean adjoint on the macroscopic space.

\begin{lemma}[Equilibrium lifting]
\label{lem:lifting}
The lifting $\Klift$ is an isometry and
\begin{equation}
 \Klift^\adj=\Smom,\qquad
 \norm{\Klift w}_D=\norm{w}_2,\qquad
 \norm{\Smom}_{D\to2}=1.
 \label{eq:lifting-isometry}
\end{equation}
\end{lemma}

\begin{proof}
By symmetry of the $K_i$ and \eqref{eq:sumK},
\[
 \norm{\Klift w}_D^2
 =\sum_{x,i}w(x)^TK_iw(x)
 =\sum_x\norm{w(x)}_2^2.
\]
Moreover,
\[
 \ip{\Klift w}{f}_D
 =\sum_{x,i}w(x)^Tf_i(x)
 =\ip{w}{\Smom f}_2.
\]
This gives the adjoint identity and the norm statements.
\end{proof}

\subsection{An abstract transport class}

\begin{definition}[Reversible equilibrium-compatible transport]
\label{def:reversible}
A homogeneous (zero-input) transport operator $\Trans:\Hpop_\Lambda\to\Hpop_\Lambda$ is called reversible and equilibrium-compatible if
\begin{equation}
 \Trans^\adj\Trans=I,
 \qquad
 \Trans^{-1}=\Jrev\Trans\Jrev.
 \label{eq:transport-assumptions}
\end{equation}
Since the population space is finite dimensional, the first identity makes
$\Trans$ invertible with $\Trans^{-1}=\Trans^\adj$ and hence unitary in
\eqref{eq:Dmetric}; the second identity is time reversal by exchange of
opposite velocities.
\end{definition}

The following linkwise criterion covers the boundary rules used below.  It also makes clear that the theorem is not restricted to a rectangular domain.

\begin{lemma}[Linkwise reflection criterion]
\label{lem:linkwise}
Assume that the transport matrix has exactly one nonzero invertible $m\times m$
block in every block row and every block column.  Thus it is a block
permutation.  For every directed block link that sends the source degree
$(y,j)$ to the target degree $(x,i)$ with component block $Q$, require
\begin{equation}
 Q^TK_i^{-1}Q=K_j^{-1}.
 \label{eq:link-metric}
\end{equation}
Require also that the source $(x,\bar i)$ is sent to the target $(y,\bar j)$ with block $Q^{-1}$.  Interior and periodic links use the identity component map.  A missing predecessor link for velocity $i$ is closed by
\begin{equation}
 f_i^{n+1}(x)=R_{x,i}f_{\bar i}^{*,n}(x),
 \label{eq:link-reflection}
\end{equation}
where
\begin{equation}
 R_{x,i}^{-1}=R_{x,i},\qquad
 R_{x,i}^TK_i^{-1}R_{x,i}=K_{\bar i}^{-1}.
 \label{eq:reflection-metric}
\end{equation}
Then the resulting global transport satisfies \eqref{eq:transport-assumptions}.
\end{lemma}

\begin{proof}
For a link $(y,j)\to(x,i)$, equation~\eqref{eq:link-metric} gives
\[
 (Qv)^TK_i^{-1}(Qv)=v^TK_j^{-1}v.
\]
Summing this equality once over all block columns, which is possible because
the block-permutation map is bijective, shows that the population energy is
unchanged and $\Trans^\adj\Trans=I$.  If
$\Trans_{(x,i),(y,j)}=Q$, the paired-link assumption says
$\Trans_{(y,\bar j),(x,\bar i)}=Q^{-1}$.  Consequently the
$((y,j),(x,i))$ block of $\Jrev\Trans\Jrev$ is $Q^{-1}$, exactly the
corresponding block of $\Trans^{-1}$.  This proves the reversal identity.
For a reflected link, the two requirements reduce to
\eqref{eq:reflection-metric} and $R_{x,i}^{-1}=R_{x,i}$.
\end{proof}

The preceding block-permutation criterion is convenient but not necessary.  The next criterion treats a boundary patch as one multi-channel scattering block.

\begin{lemma}[Boundary-patch scattering criterion]
\label{lem:patch-scattering}
Partition the boundary source and target degrees of freedom into disjoint patches.  On a patch $\beta$, stack the outgoing source blocks in $g_\beta\in\mathcal H_{O_\beta}$ and the missing incoming target blocks in $h_\beta\in\mathcal H_{I_\beta}$, and impose
\begin{equation}
 h_\beta=U_\beta g_\beta,
 \qquad U_\beta:\mathcal H_{O_\beta}\longrightarrow\mathcal H_{I_\beta}.
 \label{eq:patch-scattering}
\end{equation}
Let $D_{O_\beta}$ and $D_{I_\beta}$ be the restrictions of $D_\Lambda$ to these stacked blocks.  Suppose velocity reversal bijectively pairs $O_\beta$ with $I_\beta$, and denote its restriction from $\mathcal H_{O_\beta}$ to $\mathcal H_{I_\beta}$ by
\[
 J_\beta:\mathcal H_{O_\beta}\longrightarrow\mathcal H_{I_\beta}.
\]
If
\begin{equation}
 U_\beta^T D_{I_\beta}U_\beta=D_{O_\beta},
 \qquad
 U_\beta^{-1}=J_\beta^{-1}U_\beta J_\beta^{-1},
 \label{eq:patch-conditions}
\end{equation}
for every patch, and all remaining source--target links occur in reverse-paired metric-compatible pairs as in Lemma~\ref{lem:linkwise}, then the resulting global transport satisfies Definition~\ref{def:reversible}.  The second condition in \eqref{eq:patch-conditions} is equivalently
\[
 (J_\beta^{-1}U_\beta)^2=I_{\mathcal H_{O_\beta}}.
\]
\end{lemma}

\begin{proof}
The source and target partitions are disjoint and exhaustive.  Hence the first identity in \eqref{eq:patch-conditions}, summed over patches and combined with \eqref{eq:link-metric} on the remaining links, gives
\[
 \norm{\Trans f}_D^2=\norm{f}_D^2.
\]
The reversal bijection $J_\beta$ makes the source and target dimensions equal.  The first identity in \eqref{eq:patch-conditions} then makes $U_\beta$ injective and hence invertible, while the summed identity makes the finite-dimensional global map $D$-unitary.  On $\mathcal H_{I_\beta}$, the restriction of $\Jrev\Trans\Jrev$ is $J_\beta^{-1}U_\beta J_\beta^{-1}:\mathcal H_{I_\beta}\to\mathcal H_{O_\beta}$, which equals the corresponding restriction of $\Trans^{-1}$ by the second identity.  The reverse-paired interior links give the same equality off the boundary patches.  Thus $\Trans^{-1}=\Jrev\Trans\Jrev$ globally.
\end{proof}

The criteria are realized by the following transports.

\paragraph{Periodic transport.}
Pairwise periodic identification in each lattice direction is a block permutation with identity component blocks.  It preserves the equilibrium metric, and velocity reversal interchanges each periodic source--target pair; hence it satisfies Definition~\ref{def:reversible}.

\paragraph{Vector no-slip halfway bounce-back.}
At every missing predecessor link, take
\begin{equation}
 R_{x,i}=R_0=\diag(1,-I_d).
 \label{eq:noslip-reflection}
\end{equation}
Equation \eqref{eq:Kconjugacy}, after taking inverses and using
$R_0^T=R_0^{-1}=R_0$, gives \eqref{eq:reflection-metric}.  This is the homogeneous vector-type halfway rule analyzed in the one-rate setting in \cite{ZhaoZhangYong2020}.

\paragraph{Coordinate-aligned specular reflection.}
Suppose the missing link is normal to the coordinate direction $e_j$.  Let
\begin{equation}
 R_j=\diag\bigl(1,I_d-2e_je_j^T\bigr).
 \label{eq:specular-reflection}
\end{equation}
This map reverses the normal momentum component and leaves tangential components unchanged.  Since $R_jB_jR_j=-B_j$, it satisfies
\[
 R_jK_{+j}R_j=K_{-j},
\]
Taking inverses and using $R_j^T=R_j^{-1}=R_j$ therefore gives
\eqref{eq:reflection-metric}.  The same proof applies to mixtures of periodic, no-slip, and specular links as long as the global source-to-target map remains bijective.

\paragraph{Tangential orthogonal involutions.}
The two preceding wall reflections are endpoints of a larger componentwise class.  Let $Q_j:e_j^\perp\to e_j^\perp$ be any orthogonal involution and extend it by zero on $\operatorname{span}\{e_j\}$.  On $\R^d$ this extension satisfies
\[
 Q_je_j=0,\qquad Q_j^T=Q_j,\qquad
 Q_j^2=I_d-e_je_j^T.
\]
Set
\begin{equation}
 S_{j,Q}=-e_je_j^T+Q_j,
 \qquad
 R_{j,Q}=\diag(1,S_{j,Q}).
 \label{eq:tangential-involution}
\end{equation}
Then $S_{j,Q}^2=I_d$, $R_{j,Q}^2=I_m$, and
\[
 R_{j,Q}B_jR_{j,Q}=-B_j.
\]
Consequently $R_{j,Q}K_{+j}R_{j,Q}=K_{-j}$ and \eqref{eq:reflection-metric} holds.  The choices $Q_j=-(I_d-e_je_j^T)$ and $Q_j=I_d-e_je_j^T$ give vector no-slip and specular reflection, respectively.  In three dimensions, other orthogonal involutions of the two-dimensional tangential plane give additional reversible component reflections.

\paragraph{Patchwise orthogonal mixing.}
Consider $k$ cut links with the same coordinate normal $e_j$ and the same component reflection $R_{j,Q}$.  Use the canonical outgoing--incoming ordering, in which the matrix of the restricted velocity exchange is $J_\beta=I_{km}$, and let
\begin{equation}
 U_\beta=O_\beta\otimes R_{j,Q},
 \qquad O_\beta^TO_\beta=I_k,\qquad O_\beta^2=I_k.
 \label{eq:patch-orthogonal}
\end{equation}
Here $O_\beta$ may permute or orthogonally mix the $k$ wall channels.  Since
\[
 D_{I_\beta}=I_k\otimes K_{+j}^{-1},
 \qquad D_{O_\beta}=I_k\otimes K_{-j}^{-1},
\]
 equations \eqref{eq:tangential-involution}--\eqref{eq:patch-orthogonal} give the weighted-isometry identity in \eqref{eq:patch-conditions}.  They also give $U_\beta^{-1}=U_\beta=J_\beta^{-1}U_\beta J_\beta^{-1}$, so the reversal identity holds as well.  Lemma~\ref{lem:patch-scattering} therefore covers such finite multi-channel wall scattering.  This construction is a homogeneous stability class; no off-midpoint accuracy claim is attached to it.

\subsection{Macroscopic compression}

For a transport satisfying Definition~\ref{def:reversible}, define
\begin{equation}
 P=\Smom\Trans\Klift:\Wmac_\Lambda\to\Wmac_\Lambda.
 \label{eq:P}
\end{equation}

\begin{lemma}[Contraction and adjoint]
\label{lem:P}
The operator $P$ is a contraction in the Euclidean macroscopic norm, and
\begin{equation}
 P^*=\Smom\Trans\Jrev\Klift.
 \label{eq:P-adjoint}
\end{equation}
\end{lemma}

\begin{proof}
Lemma~\ref{lem:lifting} and transport unitarity give
\[
 \norm{Pw}_2\le\norm{\Trans\Klift w}_D=\norm{w}_2.
\]
For the adjoint,
\begin{align*}
 P^*
 &=\Klift^\adj\Trans^\adj\Klift
  =\Smom\Trans^{-1}\Klift
  =\Smom\Jrev\Trans\Jrev\Klift\\
 &=\Smom\Trans\Jrev\Klift,
\end{align*}
where $\Smom\Jrev=\Smom$ was used in the last step.
\end{proof}

\section{A mesh-uniform companion estimate}
\label{sec:companion}

This section proves a power bound, uniform in the dimension, for the companion recurrence associated with an arbitrary contraction.  We work on a finite-dimensional complex Hilbert space.  Real operators are complexified without changing their operator norms or the power estimates below.  Inner products are linear in the first argument, and $^*$ denotes the Hilbert-space adjoint in this section.  For a bounded operator $A$, let
\[
 W(A)=\{\ip{Ax}{x}:\norm{x}=1\}
\]
denote its numerical range.  We use the following consequence of the Crouzeix--Palencia theorem \cite{CrouzeixPalencia2017}: for every polynomial $q$,
\begin{equation}
 \norm{q(A)}\le \kappa_{\rm CP}
 \sup_{z\in W(A)}|q(z)|,
 \qquad \kappa_{\rm CP}=1+\sqrt2.
 \label{eq:CP}
\end{equation}
The same estimate holds with the closure of the numerical range, which is immaterial below.

Let $P$ be a contraction and $b\in\R$ with $|b|<1$.  Set
\begin{equation}
 S_b(P)=P+bP^*,
 \qquad
 \mathscr E_b=\{\zeta+b\overline\zeta:|\zeta|\le1\}.
 \label{eq:ellipse}
\end{equation}
The set $\mathscr E_b$ is the closed ellipse
\[
 \left\{(1+b)x+\mathrm{i}(1-b)y:x^2+y^2\le1\right\}.
\]

\begin{lemma}[Numerical-range inclusion]
\label{lem:ellipse}
If $P$ is a contraction, then
\begin{equation}
 W(S_b(P))\subset\mathscr E_b.
 \label{eq:W-ellipse}
\end{equation}
\end{lemma}

\begin{proof}
For a unit vector $x$, set $p=\ip{Px}{x}$.  Then $|p|\le1$ and, by conjugate symmetry,
\[
 \ip{S_b(P)x}{x}=p+b\overline p\in\mathscr E_b.
\]
\end{proof}

Define scalar polynomials by
\begin{equation}
 q_0(s)=0,\qquad q_1(s)=1,\qquad
 q_{n+1}(s)=s q_n(s)-bq_{n-1}(s),\quad n\ge1.
 \label{eq:q-recurrence}
\end{equation}

\begin{lemma}[Uniform polynomial bound on the ellipse]
\label{lem:q-bound}
For every $n\ge1$,
\begin{equation}
 \sup_{s\in\mathscr E_b}|q_n(s)|
 \le \frac{2}{1-|b|}.
 \label{eq:q-bound}
\end{equation}
\end{lemma}

\begin{proof}
The ellipse has nonempty interior because $|b|<1$.  Its boundary is parameterized by
\[
 s=\zeta+b\overline\zeta,
 \qquad |\zeta|=1.
\]
For such $s$, the two roots of $r^2-sr+b=0$ are $\zeta$ and $b\overline\zeta$.  They are distinct because $|b|<1$.  The recurrence \eqref{eq:q-recurrence} therefore gives
\begin{equation}
 q_n(s)=
 \frac{\zeta^n-(b\overline\zeta)^n}
      {\zeta-b\overline\zeta},
 \qquad n\ge1.
 \label{eq:q-boundary-formula}
\end{equation}
The denominator satisfies
\[
 |\zeta-b\overline\zeta|
 =|\zeta^2-b|\ge1-|b|,
\]
whereas the numerator is at most $1+|b|^n\le2$.  Thus \eqref{eq:q-bound} holds on the boundary.  Since $q_n$ is entire and the ellipse has nonempty interior, the maximum-modulus principle extends it to the whole ellipse.
\end{proof}

\begin{theorem}[Uniform companion estimate for contractions]
\label{thm:companion}
Let $\Wmac$ be a finite-dimensional complex Hilbert space, let $P:\Wmac\to\Wmac$ be a contraction, and let $b\in\R$ with $|b|<1$.  Define
\begin{equation}
 G_b(P)=
 \begin{pmatrix}
 P+bP^*&-bI\\
 I&0
 \end{pmatrix}
 \quad\text{on }\Wmac\oplus\Wmac.
 \label{eq:G}
\end{equation}
Then
\begin{equation}
 \sup_{n\ge0}\norm{G_b(P)^n}
 \le C_G(b),
 \qquad
 C_G(b)=
 \frac{2(1+\sqrt2)\sqrt{2(1+b^2)}}{1-|b|}.
 \label{eq:G-uniform-bound}
\end{equation}
In particular, the bound is independent of the dimension of $\Wmac$ and of the particular contraction $P$.
\end{theorem}

\begin{proof}
Put $S=S_b(P)$.  For $n=1$, the following formula follows from
$q_0(S)=0$, $q_1(S)=I$, and $q_2(S)=S$.  If it holds at $n$, left
multiplication by $G_b(P)$ changes the first block row to
\[
 \bigl(Sq_{n+1}(S)-bq_n(S),
 -b[S q_n(S)-bq_{n-1}(S)]\bigr)
 =\bigl(q_{n+2}(S),-bq_{n+1}(S)\bigr),
\]
while the old first row becomes the second row.  Thus induction using
\eqref{eq:q-recurrence} gives, for $n\ge1$,
\begin{equation}
 G_b(P)^n=
 \begin{pmatrix}
 q_{n+1}(S)&-bq_n(S)\\
 q_n(S)&-bq_{n-1}(S)
 \end{pmatrix}.
 \label{eq:G-power-formula}
\end{equation}
By Lemma~\ref{lem:ellipse}, Lemma~\ref{lem:q-bound}, and \eqref{eq:CP},
\begin{equation}
 \norm{q_n(S)}
 \le Q_b:=\frac{2(1+\sqrt2)}{1-|b|},
 \qquad n\ge1.
 \label{eq:q-operator-bound}
\end{equation}
The same inequality is valid for $q_0(S)=0$.  For $(x,y)\in\Wmac\oplus\Wmac$, formula \eqref{eq:G-power-formula} gives
\begin{align*}
 \norm{G_b(P)^n(x,y)}^2
 &\le 2Q_b^2\bigl(\norm{x}+|b|\norm{y}\bigr)^2\\
 &\le 2Q_b^2(1+b^2)\bigl(\norm{x}^2+\norm{y}^2\bigr).
\end{align*}
This proves \eqref{eq:G-uniform-bound} for $n\ge1$.  Since $C_G(b)>1$, the case $n=0$ is included.
\end{proof}

\begin{remark}
The constant in \eqref{eq:G-uniform-bound} is not claimed to be sharp.  Its deterioration as $|b|\to1$ is consistent with the loss of damping at the relaxation endpoints.  The point needed below is that no lattice dimension, boundary geometry, or spectral separation enters the estimate.
\end{remark}

\section{Mesh-uniform power and resolvent stability}
\label{sec:stability}

This section combines the macroscopic recurrence with the companion estimate, lifts the resulting bound to the full population update, and derives a uniform exterior resolvent bound and a weighted $\ell^2$ estimate for inhomogeneous terms.

Let $\Trans$ satisfy Definition~\ref{def:reversible}.  The complete homogeneous linearized update is
\begin{equation}
 f^{n+1}=\Trans\Coll f^n,
 \qquad w^n=\Smom f^n.
 \label{eq:full-update}
\end{equation}
Under \eqref{eq:otrt}, equations \eqref{eq:collision-otrt} and \eqref{eq:full-update} give
\begin{equation}
 f^{n+1}=\Trans(I+b\Jrev)\Klift w^n-b\Trans\Jrev f^n.
 \label{eq:full-otrt}
\end{equation}
Define
\begin{equation}
 \mathcal R=\Trans\Jrev.
 \label{eq:R}
\end{equation}
The reversibility identity implies
\begin{equation}
 \mathcal R^2=\Trans\Jrev\Trans\Jrev=I,
 \qquad
 \Trans\Jrev\Trans=\Jrev.
 \label{eq:R-involution}
\end{equation}

\begin{lemma}[Exact macroscopic recurrence]
\label{lem:recurrence}
Every solution of \eqref{eq:full-otrt} satisfies, for $n\ge1$,
\begin{equation}
 w^{n+1}=(P+bP^*)w^n-bw^{n-1}.
 \label{eq:recurrence}
\end{equation}
\end{lemma}

\begin{proof}
Applying $\Smom$ to \eqref{eq:full-otrt} and using \eqref{eq:P} and \eqref{eq:P-adjoint} gives
\begin{equation}
 w^{n+1}=Pw^n+bP^*w^n-b\Smom\mathcal R f^n.
 \label{eq:recurrence-pre}
\end{equation}
At the preceding time level,
\[
 f^n=\Trans(I+b\Jrev)\Klift w^{n-1}-b\Trans\Jrev f^{n-1}.
\]
Multiplication by $\mathcal R=\Trans\Jrev$ and use of \eqref{eq:R-involution} yield
\begin{align*}
 \mathcal Rf^n
 &=\Trans\Jrev\Trans(I+b\Jrev)\Klift w^{n-1}
   -b(\Trans\Jrev)^2f^{n-1}\\
 &=(\Jrev+bI)\Klift w^{n-1}-bf^{n-1}.
\end{align*}
Because $\Smom\Jrev=\Smom$ and $\Smom\Klift=I$,
\[
 \Smom\mathcal R f^n=(1+b)w^{n-1}-bw^{n-1}=w^{n-1}.
\]
Substitution into \eqref{eq:recurrence-pre} proves \eqref{eq:recurrence}.
\end{proof}

The next local constants are independent of the number of lattice nodes.  Set
\begin{equation}
 \kappa_-=
 \min\left\{a-\frac\alpha2,\,1-2da\right\},
 \qquad
 \kappa_+=
 \max\left\{a+\frac\alpha2,\,1-2da\right\},
 \label{eq:kappa-pm}
\end{equation}
and
\begin{equation}
 c_J=\left(\frac{a+\alpha/2}{a-\alpha/2}\right)^{1/2}.
 \label{eq:cJ}
\end{equation}

\begin{lemma}[Uniform local norm bounds]
\label{lem:local-uniform}
Under \eqref{eq:positive-range},
\begin{equation}
 \frac1{\sqrt{\kappa_+}}\norm{f}_2
 \le\norm{f}_D
 \le\frac1{\sqrt{\kappa_-}}\norm{f}_2,
 \qquad
 \norm{\Jrev f}_D\le c_J\norm{f}_D.
 \label{eq:norm-equivalence-J}
\end{equation}
The constants are independent of $\Lambda$.
\end{lemma}

\begin{proof}
The first two inequalities follow by applying the smallest and largest eigenvalue bounds for the local matrices $K_i$ at every node.  For the reversal estimate, rename opposite velocity indices to obtain
\[
 \norm{\Jrev f}_D^2
 =\sum_{x,i}f_i(x)^TK_{\bar i}^{-1}f_i(x).
\]
For a moving pair, $K_i$ and $K_{\bar i}$ commute.  On
$\operatorname{span}\{\mathbf e_0,\mathbf e_j\}$, the eigenvalues of
$K_iK_{\bar i}^{-1}$ are $r$ and $r^{-1}$, where
$r=(a+\alpha/2)/(a-\alpha/2)$; on the orthogonal complement and the rest
block the ratio is $1$.  Summation gives the second assertion.
\end{proof}

\begin{theorem}[Mesh-uniform power stability]
\label{thm:main}
Let $d\in\{1,2,3\}$ and assume
\begin{equation}
 0<\alpha<2a<\frac1d,
 \qquad 0<s_-<2,
 \qquad s_+=2-s_-.
 \label{eq:main-conditions}
\end{equation}
For every nonempty finite lattice $\Lambda\subset\mathbb Z^d$ and every transport $\Trans$ satisfying Definition~\ref{def:reversible}, the complete rest-state linearized amplification operator
\[
 A_\Lambda=\Trans\Coll
\]
is power bounded uniformly in $\Lambda$ and $\Trans$.  More precisely, there exists an explicit constant
\[
 M_D=M_D(a,\alpha,s_-)<\infty
\]
such that
\begin{equation}
 \sup_{n\ge0}\norm{A_\Lambda^n}_{D\to D}\le M_D,
 \label{eq:D-uniform-power}
\end{equation}
and therefore
\begin{equation}
 \sup_{n\ge0}\norm{A_\Lambda^n}_{2\to2}
 \le \left(\frac{\kappa_+}{\kappa_-}\right)^{1/2}M_D.
 \label{eq:2-uniform-power}
\end{equation}
The same bounds hold for grid-weighted $\ell^2$ norms obtained by multiplying both population norms by a common cell-volume factor.
\end{theorem}

\begin{proof}
Let $\beta=|b|<1$ and define
\begin{equation}
 c_C=1+2\beta c_J,
 \qquad
 c_L=1+\beta c_J,
 \qquad
 M_w=C_G(b)\sqrt{1+c_C^2}.
 \label{eq:proof-constants}
\end{equation}
Because $\Klift$ is an isometry and $\Smom=\Klift^\adj$ is a contraction, the equilibrium projection $E=\Klift\Smom$ is an orthogonal projection in the $D$-metric.  From \eqref{eq:collision-otrt} and Lemma~\ref{lem:local-uniform},
\begin{align}
 \norm{\Coll f}_D
 &\le \norm{Ef}_D+\beta\norm{\Jrev Ef}_D
       +\beta\norm{\Jrev f}_D\notag\\
 &\le c_C\norm{f}_D.
 \label{eq:C-uniform}
\end{align}
Thus
\[
 \norm{w^0}_2=\norm{\Smom f^0}_2\le\norm{f^0}_D,
 \qquad
 \norm{w^1}_2=\norm{\Smom\Trans\Coll f^0}_2
 \le c_C\norm{f^0}_D.
\]
Lemma~\ref{lem:P} makes $P$ a contraction.  For $n\ge1$,
Lemma~\ref{lem:recurrence} gives
\[
 \binom{w^n}{w^{n-1}}
 =G_b(P)^{n-1}\binom{w^1}{w^0}.
\]
Theorem~\ref{thm:companion}, together with the two preceding initial bounds,
therefore gives
\begin{equation}
 \sup_{n\ge0}\norm{w^n}_2
 \le M_w\norm{f^0}_D.
 \label{eq:w-uniform}
\end{equation}

Set
\[
 L=\Trans(I+b\Jrev)\Klift.
\]
Transport unitarity, Lemma~\ref{lem:lifting}, and Lemma~\ref{lem:local-uniform} imply
\begin{equation}
 \norm{L}_{2\to D}\le c_L.
 \label{eq:L-uniform}
\end{equation}
Equation \eqref{eq:full-otrt} can be written
\begin{equation}
 f^{n+1}=Lw^n-b\mathcal Rf^n.
 \label{eq:variation-step}
\end{equation}
Iteration gives
\begin{equation}
 f^n=(-b\mathcal R)^nf^0
 +\sum_{k=0}^{n-1}(-b\mathcal R)^{n-1-k}Lw^k.
 \label{eq:Duhamel}
\end{equation}
Since $\mathcal R^2=I$, one has $\mathcal R^{2k}=I$ and
$\mathcal R^{2k+1}=\Trans\Jrev$ for $k\ge0$.  Transport unitarity and
Lemma~\ref{lem:local-uniform} give
\begin{equation}
 \sup_{j\ge0}\norm{\mathcal R^j}_{D\to D}
 \le c_J.
 \label{eq:R-uniform}
\end{equation}
Combining \eqref{eq:w-uniform}--\eqref{eq:R-uniform} with the geometric series yields
\begin{equation}
 \norm{f^n}_D
 \le c_J\left(1+\frac{c_LM_w}{1-\beta}\right)\norm{f^0}_D.
 \label{eq:MD-formula}
\end{equation}
Thus one admissible explicit constant is
\begin{equation}
 M_D=
 c_J\left[
 1+\frac{(1+\beta c_J)C_G(b)
 \sqrt{1+(1+2\beta c_J)^2}}
 {1-\beta}
 \right].
 \label{eq:MD}
\end{equation}
It depends only on $a$, $\alpha$, and $s_-$.  This proves \eqref{eq:D-uniform-power}.  The Euclidean estimate \eqref{eq:2-uniform-power} follows from \eqref{eq:norm-equivalence-J}.  Multiplication of both norms by a common cell-volume factor leaves every operator estimate unchanged.
\end{proof}

\begin{corollary}[Refining mesh families]
\label{cor:mesh-family}
Fix $d$, $a$, $\alpha$, and $s_-$ satisfying \eqref{eq:main-conditions}.  Let $h$ range over any set of positive mesh widths, let $\Lambda_h\subset\mathbb Z^d$ be finite and nonempty, and let $\Trans_h$ satisfy Definition~\ref{def:reversible}.  Define
\[
 \norm{f}_{D,h}^2
 =h^d\sum_{x\in\Lambda_h}\sum_{i\in\mathcal V}
 f_i(x)^TK_i^{-1}f_i(x).
\]
Then, with $A_h=\Trans_h\Coll$,
\[
 \sup_h\sup_{n\ge0}\norm{A_h^n}_{D,h\to D,h}\le M_D.
\]
Thus the estimate remains uniform along arbitrary refining lattice families.
More generally, let $\Theta$ be a compact subset of the open parameter region
in \eqref{eq:main-conditions}, and suppose that
$(a_h,\alpha_h,s_{-,h})\in\Theta$.  Let $K_{i,h}$, $D_h$, and $\Coll_h$ be the
associated equilibrium blocks, metric, and collision; write
$\norm{f}_{D_h,h}^2:=h^d\sum_{x,i}f_i(x)^TK_{i,h}^{-1}f_i(x)$, and suppose that
$\Trans_h$ satisfies Definition~\ref{def:reversible} relative to $D_h$.  Then,
with $A_h=\Trans_h\Coll_h$,
\[
 M_\Theta:=\max_{(a,\alpha,s_-)\in\Theta}M_D(a,\alpha,s_-)<\infty,
 \qquad
 \sup_h\sup_{n\ge0}\norm{A_h^n}_{D_h,h\to D_h,h}\le M_\Theta.
\]
\end{corollary}

\begin{proof}
The factor $h^d$ multiplies both sides of every population norm estimate and hence leaves the induced operator norm unchanged.  The explicit function $M_D$ in \eqref{eq:MD} is continuous on the admissible parameter region, so it attains the finite maximum $M_\Theta$ on $\Theta$.
\end{proof}

\begin{corollary}[Concrete reversible boundaries]
\label{cor:boundaries}
Under \eqref{eq:main-conditions}, the uniform estimates in Theorem~\ref{thm:main} apply to arbitrary finite lattices equipped with any bijective compatible mixture of
\begin{enumerate}[label=(\alph*),leftmargin=2.5em]
\item periodic transport;
\item vector no-slip halfway bounce-back \eqref{eq:noslip-reflection};
\item coordinate-aligned specular reflection \eqref{eq:specular-reflection};
\item tangential orthogonal involutions \eqref{eq:tangential-involution};
\item patchwise orthogonal mixing \eqref{eq:patch-orthogonal}.
\end{enumerate}
For item (e), the patch partition is required to satisfy Lemma~\ref{lem:patch-scattering}.  The constant is independent of the number and arrangement of lattice nodes.
\end{corollary}

\begin{proof}
Items (a)--(d) satisfy Lemma~\ref{lem:linkwise}; item (e) satisfies Lemma~\ref{lem:patch-scattering}.  Theorem~\ref{thm:main} then applies.
\end{proof}

\begin{corollary}[Uniform resolvent and weighted $\ell^2$ estimate]
\label{cor:resolvent}
Under the assumptions of Theorem~\ref{thm:main}, for every $z\in\C$ with $|z|>1$,
\begin{equation}
 \norm{(zI-A_\Lambda)^{-1}}_{D\to D}
 \le \frac{M_D}{|z|-1}.
 \label{eq:resolvent}
\end{equation}
If
\[
 f^{n+1}=A_\Lambda f^n+g^n,
 \qquad f^0=0,
\]
then, for every $r>1$,
\begin{equation}
 \left(\sum_{n\ge1}r^{-2n}\norm{f^n}_D^2\right)^{1/2}
 \le \frac{M_D}{r-1}
 \left(\sum_{n\ge0}r^{-2n}\norm{g^n}_D^2\right)^{1/2}.
 \label{eq:weighted-input}
\end{equation}
Both estimates are uniform in $\Lambda$ and $\Trans$.
\end{corollary}

\begin{proof}
The series
\[
 (zI-A_\Lambda)^{-1}
 =\sum_{n\ge0}z^{-n-1}A_\Lambda^n
\]
converges absolutely in operator norm by \eqref{eq:D-uniform-power} and gives \eqref{eq:resolvent}.  For the recursion with zero initial state,
\[
 f^n=\sum_{k=0}^{n-1}A_\Lambda^{n-1-k}g^k.
\]
After multiplication by $r^{-n}$, this is a discrete convolution whose operator-valued kernel is $r^{-j-1}A_\Lambda^j$, $j\ge0$.  Its $\ell^1$ norm is at most $M_D/(r-1)$.  Young's convolution inequality gives \eqref{eq:weighted-input}.
\end{proof}

\begin{remark}[Relation to GKS stability]
\label{rem:GKS}
Theorem~\ref{thm:main} and Corollary~\ref{cor:resolvent} provide a mesh-uniform homogeneous state estimate, an exterior Kreiss-type resolvent bound, and an exponentially weighted $\ell^2$ estimate for inhomogeneous forcing.  They are not a GKS strong-stability theorem: in particular, they do not give a sharp boundary-trace estimate for arbitrary incoming data or analyze a boundary symbol for the off-midpoint interpolation of Section~\ref{sec:interpolation}.  Those questions require additional boundary-specific work.
\end{remark}

\begin{remark}[Scope]
The uniform estimates concern the homogeneous rest-state linearization on the OTRT line.  The $D$-norm constant $M_D$ deteriorates as $s_-$ approaches $0$ or $2$, or as $a-\alpha/2$ approaches zero.  If $1-2da$ approaches zero while the moving-block gap stays positive, the additional Euclidean conversion factor in \eqref{eq:2-uniform-power}, rather than $M_D$, deteriorates.  A nonzero uniform background velocity changes the equilibrium Jacobians.  Nonlinear stability and convergence to a target PDE are not consequences of the present theorem.
\end{remark}

\section{The off-midpoint parameterized boundary}
\label{sec:interpolation}

The reversible-transport theorem does not automatically cover boundary rules that mix pre- and postcollision populations.  This section analyzes one such off-midpoint interpolation and gives an exact admissible counterexample showing that coefficient convexity alone is not an unconditional stability criterion.

For a boundary-adjacent node and an incoming lattice direction whose predecessor lies outside the fluid domain, consider the homogeneous three-coefficient vector formula obtained by combining the vector reflection with the convex single-node interpolation coefficients of \cite{ZhaoYong2017,ZhaoHuangYong2019}:
\begin{equation}
\begin{aligned}
 f_i^{n+1}(x)={}&\frac{\ell}{1+\ell}f_i^{*,n}(x)
 +\frac{1+\ell-2\gamma}{1+\ell}R_0f_{\bar i}^{n}(x)\\
 &+\frac{2\gamma-\ell}{1+\ell}R_0f_{\bar i}^{*,n}(x),
\end{aligned}
 \label{eq:parameterized-boundary}
\end{equation}
where
\begin{equation}
 0<\gamma\le1,
 \qquad
 \max\{0,2\gamma-1\}\le\ell\le2\gamma.
 \label{eq:parameter-range}
\end{equation}
Here $\gamma$ is the normalized distance from the fluid node to the wall along the cut link, and $\ell$ is the interpolation parameter.  The three coefficients are then nonnegative and sum to one.  The pure halfway vector bounce-back rule is the special case $(\gamma,\ell)=(1/2,0)$.

Formula \eqref{eq:parameterized-boundary} genuinely mixes pre- and postcollision values except on the degeneracy line $\ell=2\gamma-1$, $1/2\le\gamma\le1$.  The halfway point $(1/2,0)$ is covered by Lemma~\ref{lem:linkwise}.  Off the degeneracy line the boundary relation is not a collision-independent postcollision scattering on the original population state.  On the line it is postcollision-only, but the same postcollision source is used both by ordinary forward streaming and by the boundary interpolation.  The following proposition makes the resulting failure of metric unitarity exact.

\begin{proposition}[Characterization of the postcollision-only subfamily]
\label{prop:postcollision-intersection}
Assume that a boundary-adjacent node $x$ has a forward fluid neighbor $x+i\in\Lambda$ and that $f_i^*(x)$ undergoes ordinary forward streaming to that neighbor.  On the degeneracy line
\[
 \ell=2\gamma-1,\qquad \frac12\le\gamma\le1,
\]
if a complete transport on the original population space uses the boundary formula \eqref{eq:parameterized-boundary} at this link, uses this ordinary interior streaming, and is $D$-unitary, then necessarily
\[
 (\gamma,\ell)=\left(\frac12,0\right).
\]
At this point the local closure is precisely vector halfway bounce-back.  If all remaining source--target links satisfy the compatible hypotheses of Lemma~\ref{lem:linkwise} or Lemma~\ref{lem:patch-scattering}, the assembled global transport satisfies Definition~\ref{def:reversible}.
\end{proposition}

\begin{proof}
On the degeneracy line, \eqref{eq:parameterized-boundary} becomes
\begin{equation}
 f_i^{n+1}(x)=A_\gamma f_i^{*,n}(x)
       +C_\gamma R_0f_{\bar i}^{*,n}(x),
 \qquad
 A_\gamma=\frac{2\gamma-1}{2\gamma},
 \quad C_\gamma=\frac1{2\gamma}.
 \label{eq:degenerate-line}
\end{equation}
Choose a postcollision array whose only nonzero source block is $f_i^*(x)=v\ne0$.  Ordinary forward streaming sends this block unchanged to $(x+i,i)$, while \eqref{eq:degenerate-line} also sends $A_\gamma v$ to $(x,i)$.  These are distinct target blocks with the same weight $K_i^{-1}$.  Hence the source contribution to the squared $D$-norm changes from
\[
 v^TK_i^{-1}v
 \quad\hbox{to}\quad
 (1+A_\gamma^2)v^TK_i^{-1}v.
\]
Metric unitarity forces $A_\gamma=0$, and therefore $\gamma=1/2$ and $\ell=0$.  Conversely, these values give $C_\gamma=1$ and the local rule \eqref{eq:noslip-reflection}, which is involutive and metric compatible.  Together with compatible remaining links, Lemma~\ref{lem:linkwise} or Lemma~\ref{lem:patch-scattering} then gives the stated global conclusion.
\end{proof}

Thus, on a standard wall patch containing such a forward fluid neighbor, the boundary formula \eqref{eq:parameterized-boundary} belongs directly to the reversible-transport class only at the halfway value.  Off the degeneracy line it uses a precollision datum; on the line Proposition~\ref{prop:postcollision-intersection} applies.  This already rules out geometry-uniform direct coverage of a non-halfway member by Theorem~\ref{thm:main}.  Here \emph{directly} means through a transport defined independently of the collision.  No assertion is made about a degenerate one-layer configuration without a forward fluid neighbor.  Nor is this a stability necessary condition for the off-midpoint family.  A rate-dependent refactorization, an augmented state, or a different stability argument would require a new proof.  Nonnegative scalar coefficients alone provide neither Definition~\ref{def:reversible} nor power stability.

The proposition concerns the entire vector population space.  An algebraic collapse in a restricted mode, such as a zero-tangential-frequency shear sector at a special relaxation rate, would not by itself verify either identity in Definition~\ref{def:reversible} for the complete transport.

The next proposition gives an exact rational counterexample.  In its full-domain operator, interior populations stream normally, the rest population remains at its node, and every missing predecessor link uses the same pair $(\gamma,\ell)$ in \eqref{eq:parameterized-boundary}.

\begin{proposition}[Exact admissible instability]
\label{prop:counterexample}
Consider D2N5 on the fully bounded $3\times3$ square.  Let
\begin{equation}
 a=\alpha=\frac15,
 \qquad
 s_-=\frac{19}{10},
 \qquad
 s_+=\frac1{10},
 \qquad
 \gamma=\frac34,
 \qquad
 \ell=1.
 \label{eq:counter-parameters}
\end{equation}
The positive-equilibrium condition and \eqref{eq:otrt} hold, and $\ell$ lies strictly inside the interval \eqref{eq:parameter-range}.  The three boundary coefficients are $1/2$, $1/4$, and $1/4$.  Nevertheless, the complete $135\times135$ amplification matrix has a real eigenvalue
\begin{equation}
 \lambda_*=1.005415113082914264\ldots>1.
 \label{eq:unstable-eigenvalue}
\end{equation}
Consequently, no unconditional power-stability theorem can hold over the full parameter set \eqref{eq:parameter-range}, even after restriction to the OTRT line and positive equilibrium blocks.
\end{proposition}

\begin{proof}
Order the velocities as $(+e_x,+e_y,-e_x,-e_y,0)$ and the vector components as $(\rho,m_x,m_y)$.  The complete matrix is assembled as $A=A_{\rm post}C+A_{\rm pre}$, where $C$ is block-diagonal collision and the two boundary matrices collect the postcollision and pre-collision coefficients in \eqref{eq:parameterized-boundary}.  All entries of $A$ are rational under \eqref{eq:counter-parameters}.  The square-domain operator commutes with the dihedral group $D_4$.  Restrict it to the symmetry subspace on which a quarter turn acts as $-1$ and vertical reflection acts as $+1$.  This invariant subspace has dimension $18$.  For a full-column-rank rational basis matrix $B\in\mathbb Q^{135\times18}$, exact reduction gives $M\in\mathbb Q^{18\times18}$ satisfying $AB=BM$.  Thus every eigenvalue of $M$ is an eigenvalue of $A$, and the characteristic polynomial of $M$ is, up to a nonzero constant,
\begin{equation}
 (10z-9)^5(10z+9)^2(40z-9)(40z+9)p_9(z),
 \label{eq:factorization}
\end{equation}
where
\begin{align}
 p_9(z)={}&32000000000z^9+4320000000z^8-84210200000z^7
 -14241277000z^6\notag\\
 &+75759196300z^5+14873855940z^4-24485318982z^3\notag\\
 &-5314983642z^2+1049800095z+231434901.
 \label{eq:p9}
\end{align}
Exact rational evaluation gives
\begin{equation}
\begin{aligned}
 p_9\left(\frac{201}{200}\right)
 &=-\frac{23587881443261289}{16000000000}<0,\\
 p_9\left(\frac{503}{500}\right)
 &=\frac{4133905199135116328391}{1953125000000000}>0.
\end{aligned}
 \label{eq:sign-bracket}
\end{equation}
The opposite signs in \eqref{eq:sign-bracket} imply that $p_9$ has a real root in $(1.005,1.006)$.  Since $p_9$ is a factor of the characteristic polynomial of $M$, this root is an eigenvalue of the full amplification matrix $A$.  A high-precision evaluation, used only to report the decimal in \eqref{eq:unstable-eigenvalue}, gives $1.005415113082914264\ldots$.  Section~SM2 (``Exact rational counterexample'') of the Supplementary Materials gives the rational reduced matrix, invariant-subspace identity, and exact factorization; Section~SM5 (``Reproduction commands'') gives the commands for verifying these identities in exact arithmetic with the accompanying programs.
\end{proof}

\begin{remark}
With the same $3\times3$ domain and the same collision parameters, homogeneous vector halfway bounce-back has spectral radius one to numerical precision and is covered by Corollary~\ref{cor:boundaries}.  This comparison rules out the bulk OTRT collision alone as the source of the unstable eigenvalue in Proposition~\ref{prop:counterexample}.
\end{remark}

\begin{remark}
Proposition~\ref{prop:counterexample} does not assert that every off-midpoint parameter is unstable.  It only excludes the full convex-coefficient range as an unconditional stability region.  The finite-domain parameter sweeps reported in Section~SM6 (``Finite-domain spectral calculations'') and Figures~SM6.1, SM6.2, and SM6.4 of the Supplementary Materials exhibit spectrally unstable parameter values and others for which no eigenvalue outside the unit disk is detected.  The sampled finite-domain pattern varies with wall fraction, interpolation parameter, and domain size.  This observation motivates boundary-symbol analysis, but it neither identifies a resonance mechanism nor proves power stability of the spectrally neutral candidates.
\end{remark}

\section{Further questions in boundary stability}
\label{sec:discussion}

We distinguish three remaining issues: additional reversible boundary operators, stability of the existing off-midpoint interpolation, and an energy-compatible off-midpoint extension.

\subsection{Additional boundary operators covered by the theorem}

The proof of Theorem~\ref{thm:main} uses only the two identities in Definition~\ref{def:reversible}.  It therefore applies beyond no-slip bounce-back whenever a boundary closure is a bijective weighted-unitary scattering of outgoing populations into incoming populations and is reversed by $\Jrev$.  Lemma~\ref{lem:linkwise} covers linkwise block permutations; Lemma~\ref{lem:patch-scattering} gives the exact weighted matrix test for multi-channel scattering.  Equations \eqref{eq:tangential-involution} and \eqref{eq:patch-orthogonal} provide nontrivial additional full-vector families satisfying that test.

This criterion also separates the present theorem from accuracy-oriented unified boundary frameworks.  The enhanced single-node ELI class \cite{MarsonEtAl2021} and the LI$^+$/EMR framework \cite{GinzburgEtAl2023Unified} include rules using several local pre- or postcollision populations, nonequilibrium corrections, different temporal discretizations, or neighboring fluid nodes.  The parametric vectorial finite-difference family in \cite{ZhangFengZhao2021} likewise uses an enlarged propagation and interpolation stencil.  None of these structural labels implies \eqref{eq:patch-conditions}.  In particular, if a source population is retained for its ordinary interior target and is also copied with a nonzero coefficient into a boundary target, the single-source argument in the proof of Proposition~\ref{prop:postcollision-intersection} rules out unitarity in the unmodified equilibrium norm.  A specific member is covered only if its complete homogeneous scattering map satisfies \eqref{eq:patch-conditions}; neighboring-node, memory, and pre/postcollision rules generally require an augmented energy or a joint collision--boundary estimate.

Rules that prescribe density, pressure, or nonzero velocity are generally affine.  Their amplification operator is governed by the homogeneous part; stability with data then requires a compatible lifting or a summability estimate as in Corollary~\ref{cor:resolvent}.  A boundary formula that genuinely mixes pre- and postcollision values is not a collision-independent postcollision transport on the original state space unless additional variables or a different factorization are introduced.  For \eqref{eq:parameterized-boundary}, such mixing occurs precisely off the line $\ell=2\gamma-1$, $1/2\le\gamma\le1$; on that line, Proposition~\ref{prop:postcollision-intersection} shows that only its halfway endpoint is $D$-unitary under ordinary forward streaming.

\subsection{Analysis of the existing off-midpoint interpolation}

The pure halfway flat-wall closure is already covered uniformly by Theorem~\ref{thm:main}.  A stable subset of the off-midpoint parameterized family may still exist.  For a flat wall, Fourier transformation in the tangential directions reduces that problem to a wall-normal recurrence.  The relevant boundary symbol must be checked for generalized eigenvalues on and outside the unit circle, including glancing and boundary--bulk resonant modes; the recent scalar analyses \cite{Bellotti2026TwoVelocities,Bellotti2026Raw} illustrate the type of information required.  On a fixed domain, spectral data must first be checked for semisimplicity of unit-circle eigenvalues.  A positive definite matrix $H$ satisfying
\begin{equation}
 A_{\gamma,\ell}^*HA_{\gamma,\ell}-H\preceq0
 \label{eq:Lyapunov}
\end{equation}
directly provides a sufficient power estimate.  Spectral radii alone are not enough when $A_{\gamma,\ell}$ is nonnormal; pseudospectra and powers should also be examined \cite{TrefethenEmbree2005}.  Such computations can suggest inequalities, but a mesh-independent theorem requires an analysis uniform in tangential wave number and wall-normal resolution.

\subsection{A possible energy-compatible off-midpoint extension}

An alternative is to build the energy structure into the boundary rule.  Introduce a finite boundary memory $z^n$ on each cut link and write an extended scattering relation
\begin{equation}
 \begin{pmatrix}f_{\mathrm{in}}^{n+1}\\ z^{n+1}\end{pmatrix}
 =\mathscr U_\gamma
 \begin{pmatrix}R_0f_{\mathrm{out}}^{*,n}\\ z^n\end{pmatrix}.
 \label{eq:extended-scattering}
\end{equation}
A candidate matrix $\mathscr U_\gamma$ should satisfy a positive energy identity or passivity inequality in an extended metric, a reversal relation, and the consistency equations that place the physical wall at the prescribed fraction $\gamma$.  These algebraic requirements can be combined with the second-order moment matching used in single-node boundary constructions \cite{ZhaoYong2017,ZhaoHuangYong2019}.  Equation \eqref{eq:extended-scattering} states algebraic constraints for an energy-compatible off-midpoint rule; no such rule is constructed here.

\section{Conclusion}
\label{sec:conclusion}

The relation $s_++s_-=2$ converts the complete vector TRT update with a reversible boundary into an operator-valued two-step recurrence.  The equilibrium metric makes the lifting isometric and the homogeneous transport unitary, so the macroscopic compression is a contraction.  Numerical-range control of the companion recurrence gives a population-space power bound that is uniform over arbitrary finite lattices and all transports in the reversible class.  The result covers periodic transport, homogeneous vector no-slip halfway bounce-back, coordinate-aligned specular reflection, tangential orthogonal involutions, and compatible multi-channel patch scattering.

The exact D2N5 construction shows that nonnegative interpolation coefficients do not provide an unconditional stability condition for the three-coefficient off-midpoint vector boundary formula considered here.  One open problem is to determine a restricted stable parameter set by a flat-wall GKS and boundary-symbol analysis; another is to construct an augmented boundary scattering rule with an imposed energy or passivity structure and then establish its consistency and PDE-level convergence.  Nonzero-background linearizations, nonlinear stability, a full boundary-trace theory for off-midpoint walls, and a completed second-order energy-compatible boundary discretization are not addressed here.

\section*{Computational verification}
The Supplementary Materials give the matrix conventions, the rational invariant-subspace calculation underlying Proposition~\ref{prop:counterexample}, and auxiliary finite-domain spectral computations.  The accompanying archive contains the exact matrices, numerical data, and programs used to reproduce these calculations and figures.  The instability assertion follows from rational identities and sign evaluations; floating-point computation is used only for the reported decimal values and the illustrative numerical experiments.

\section*{Acknowledgments}
Generative-AI tools were used for editorial and technical assistance and in organizing and testing the accompanying computational materials.  The author reviewed the resulting material and assumes responsibility for all content.

\bibliographystyle{plain}
\bibliography{references}

\clearpage
\setcounter{section}{0}
\setcounter{equation}{0}
\setcounter{figure}{0}
\setcounter{table}{0}
\renewcommand{\thesection}{SM\arabic{section}}
\renewcommand{\theequation}{SM\arabic{equation}}
\renewcommand{\thefigure}{SM\arabic{figure}}
\renewcommand{\thetable}{SM\arabic{table}}
\renewcommand{\theHsection}{SM.\arabic{section}}
\renewcommand{\theHfigure}{SM.\arabic{figure}}
\renewcommand{\theHtable}{SM.\arabic{table}}

\begin{center}
{\Large\bfseries Supplementary Materials}\par
\medskip
{\large Mesh-Uniform Power Stability of Two-Relaxation-Time Vector Lattice
Boltzmann Schemes with Reversible Boundaries}\par
\medskip
Jin Zhao
\end{center}
\bigskip

This supplement specifies the matrix assembly used in the numerical experiments, gives the exact rational calculation supporting Proposition~\ref{prop:counterexample}, and reports auxiliary finite-domain computations.  The mesh-uniform stability theorem is proved analytically in the main article.  For the counterexample, all decisive identities are rational: an invariant-subspace relation, a characteristic factorization, and two sign evaluations.  The floating-point calculations reported below are illustrative and are not used to establish the instability result.

\section{Matrix assembly}
\label{sec:sm-assembly}

For a domain $\Lambda=\{x_1,\ldots,x_N\}$, velocities are ordered as
\[
 (+e_1,\ldots,+e_d,-e_1,\ldots,-e_d,0),
\]
and each population block uses component order
\[
 (\rho,m_1,\ldots,m_d).
\]
The global vector is node-major, then velocity-major, then component-major.  The local lifting, moment, and reversal matrices are
\[
 K_{\rm loc}=\operatorname{col}(K_1,\ldots,K_q),\qquad
 S_{\rm loc}=(I_m\ \cdots\ I_m),
\]
with the reversal matrix containing an $I_m$ block in position $(i,\bar i)$.  Global matrices are obtained by Kronecker product with $I_N$.

The equilibrium weight used in the main article is
\[
 D=I_N\otimes D_{\rm loc},\qquad
 D_{\rm loc}=\operatorname{diag}\bigl(K_1^{-1},\ldots,K_q^{-1}\bigr).
\]

For the linkwise reversible rules, every target population has exactly one source.  An interior source is translated with component matrix $I_m$.  A missing link is closed with either $R_0=\operatorname{diag}(1,-I_d)$ or the coordinate reflection $R_j=\operatorname{diag}(1,I_d-2e_je_j^T)$.  The larger class in Lemma~\ref{lem:patch-scattering} permits several outgoing source blocks to be mixed into an equal number of incoming targets, subject to the two exact matrix identities in \eqref{eq:patch-conditions}.  For the numerical matrices, we evaluate the residuals of
\[
 \sum_i K_i=I,\qquad K^TDK=I,\qquad S=K^TD,
\]
\[
 T^TDT=D,\qquad T^{-1}=JTJ,\qquad (TJ)^2=I,
\]
We also form $P=STK$.  These residuals serve only to verify the matrix assembly and are not used in the analytical proof.

For the parameterized boundary, the complete matrix is assembled as
\[
 A=A_{\rm post}C+A_{\rm pre},
\]
where $C$ is block-diagonal collision.  At a missing predecessor link, the target row receives
\[
 \frac{\ell}{1+\ell}f_i^*,\qquad
 \frac{1+\ell-2\gamma}{1+\ell}R_0f_{\bar i},\qquad
 \frac{2\gamma-\ell}{1+\ell}R_0f_{\bar i}^*.
\]
This convention agrees with the incoming-direction formula in the main article.

\section{Exact rational counterexample}
\label{sec:sm-certificate}

For the D2N5 counterexample, the full matrix has dimension $9\times5\times3=135$.  All parameters and all entries are rational.  Let $\mathscr R$ be active quarter-turn rotation and $\mathscr F$ vertical reflection, acting simultaneously on nodes, velocities, and momentum components.  The matrix commutes with these actions.  The unnormalized group sum
\[
 \Pi_B=\sum_{k=0}^3(-1)^k(\mathscr R^k+\mathscr F\mathscr R^k)
\]
satisfies $\Pi_B^2=8\Pi_B$ and has rank $18$.  Thus $\Pi_B/8$ is the
projector onto the symmetry type
\[
 \mathscr Rf=-f,\qquad \mathscr Ff=f.
\]
Independent integer columns of $\Pi_B$ form a basis matrix $B$.  If $B_r$ is an invertible set of independent rows, the restricted matrix is
\[
 M=B_r^{-1}(AB)_r.
\]
All entries of $A$, $B$, and $M$ are rational.  The identity $AB=BM$ proves invariance of the range of $B$ and identifies $M$ as the restricted operator.  Factoring $\det(zI-M)$ over the rationals gives the factorization and degree-nine polynomial stated in Proposition~\ref{prop:counterexample}.  The accompanying data include the exact basis and reduced matrices, the factorization, and the values used to verify the rank, symmetry, invariant-subspace relation, polynomial coefficients, and signs.  The sign evaluations are
\[
 p_9(201/200)=-\frac{23587881443261289}{16000000000}<0,
\]
\[
 p_9(503/500)=\frac{4133905199135116328391}{1953125000000000}>0.
\]
Because the two exact values have opposite signs, the intermediate value theorem gives a real zero of $p_9$ in $(201/200,503/500)$.  Since $p_9$ is a characteristic factor of the invariant restriction and $M$ represents the restriction of $A$, this zero is an eigenvalue of the full amplification matrix and exceeds one.  No floating-point eigenvalue computation enters this conclusion; the decimal approximation in the main article is descriptive only.  The files \path{counterexample_basis.txt}, \path{counterexample_reduced_matrix.txt}, \path{counterexample_factor.txt}, and \path{counterexample_verification.txt} contain these exact quantities.

\section{Numerical verification of structural identities}
\label{sec:sm-structural}

Table~\ref{tab:structure} summarizes representative calculations.  D3N7 uses $a=\alpha=1/7$ so that $\alpha/2<a<1/(2d)$; the one- and two-dimensional cases use $a=\alpha=1/5$.  For the reversible cases in Table~\ref{tab:structure}, the largest residual among the six tested identities is below $3\times10^{-15}$.  Individual values are reported in \texttt{data/verification\_table.csv}.

\begin{table}[!htbp]
\centering
\small
\caption{Representative complete amplification spectra on the OTRT line.  Each reversible row was checked for $s_-\in\{0.2,1,1.9\}$; every reported reversible spectral radius is within $10^{-13}$ of one.}
\label{tab:structure}
\begin{tabular}{@{}llrr@{}}
\toprule
model and domain & boundary & nodes & spectral radius\\
\midrule
D1N3 interval & no-slip / specular & 7 & $1+O(10^{-14})$\\
D2N5 square & no-slip / specular & 9 & $1+O(10^{-14})$\\
D2N5 rectangle & no-slip / specular & 12 & $1+O(10^{-14})$\\
D2N5 rectangle & periodic / axis mixtures & 12 & $1+O(10^{-14})$\\
D2N5 L-shaped domain & no-slip / specular & 12 & $1+O(10^{-14})$\\
D3N7 cube & no-slip / specular & 8 & $1+O(10^{-14})$\\
D2N5 off-midpoint interpolation & parameterized & 9 & $1.005415113083$\\
\bottomrule
\end{tabular}
\end{table}

For a D2N5 calculation on the L-shaped domain with random initial data, the directly stepped macroscopic moment is compared with the recurrence in Lemma~\ref{lem:recurrence}.  In this calculation, the maximum relative residual over the tested steps and the maximum relative residual in the identity $S(TJ)f^n=w^{n-1}$ are both below $3\times10^{-16}$.  The computed norm of $P$ is one to the displayed precision.

The exact-arithmetic calculation also verifies the full-vector extension of the boundary class.  In D3N7 with $a=\alpha=1/7$, it uses the rational tangential involution and patch-channel involution
\[
 Q=O=\frac15\begin{pmatrix}3&-4\\-4&-3\end{pmatrix}.
\]
Direct rational calculation gives $R_{1,Q}^2=I$, $R_{1,Q}^TK_{+1}^{-1}R_{1,Q}=K_{-1}^{-1}$, $(O\otimes R_{1,Q})^2=I$, and
\[
 (O\otimes R_{1,Q})^T(I_2\otimes K_{+1}^{-1})(O\otimes R_{1,Q})
 =I_2\otimes K_{-1}^{-1}.
\]
For the postcollision-only line of \eqref{eq:parameterized-boundary}, forming the two-source/two-target transport block gives the following duplicated-source diagonal block of $U^TD_{\rm target}U-D_{\rm source}$ exactly:
\[
 \left(\frac{2\gamma-1}{2\gamma}\right)^2K_i^{-1}.
\]
This is the algebraic identity used in Proposition~\ref{prop:postcollision-intersection}; it vanishes for $\gamma\ge1/2$ only at $\gamma=1/2$.  These rational identities are recorded in \path{data/local_algebra_checks.json}.

For twelve randomly generated nonnormal contractions of dimensions $3$, $5$, and $8$ and $b=-0.9,-0.3,0,0.7$, we computed spectra and power norms for $0\le n\le200$.  A further example with a reducing unitary block was computed for $0\le n\le500$.  The numerical values are recorded in \path{data/recurrence_companion_checks.json} and provide comparisons with Theorem~\ref{thm:companion}; the proof itself uses the numerical-range estimate in the main article.

\section{Numerical illustrations of the uniform companion estimate}
\label{sec:sm-companion}

For the examples below, we evaluate the block identity
\[
 G_b(P)^n=
 \begin{pmatrix}
 q_{n+1}(P+bP^*)&-bq_n(P+bP^*)\\
 q_n(P+bP^*)&-bq_{n-1}(P+bP^*)
 \end{pmatrix}
\]
for random strict contractions and for truncated unilateral shifts of dimensions $4$, $8$, $16$, and $32$.  The shifts provide nonnormal examples.  The largest relative residual in the block identity is below $2\times10^{-15}$.  In every case the observed power norm lies below the explicit constant in Theorem~\ref{thm:companion}; the largest observed ratio to this nonsharp constant is approximately $0.208$.  Individual values and the theorem constants are recorded in \path{data/uniform_companion_checks.json}.  These computations illustrate the recurrence identity and the conservative size of the explicit constant; the uniform bound follows analytically from the Crouzeix--Palencia theorem and the ellipse estimate in the main article.

\section{Reproduction commands}
\label{sec:sm-reproduction}

The programs require NumPy, SymPy, and Matplotlib; the tested versions are listed in the accompanying dependency record.  From the extracted archive root, run the following command to reproduce all calculations and figures:
\begin{verbatim}
python reproduce.py
\end{verbatim}
The option \texttt{--skip-plots} performs only the algebraic calculations.  The command evaluates the exact identities and numerical examples before generating the figures.

\section{Finite-domain spectral calculations}
\label{sec:sm-interpolation}

The following computations use D2N5 with $a=\alpha=0.2$, $s_-=1.9$, and $s_+=0.1$.  They complement the exact counterexample but do not define a mesh-independent stability region.

\begin{figure}[!htbp]
\centering
\includegraphics[width=.80\textwidth]{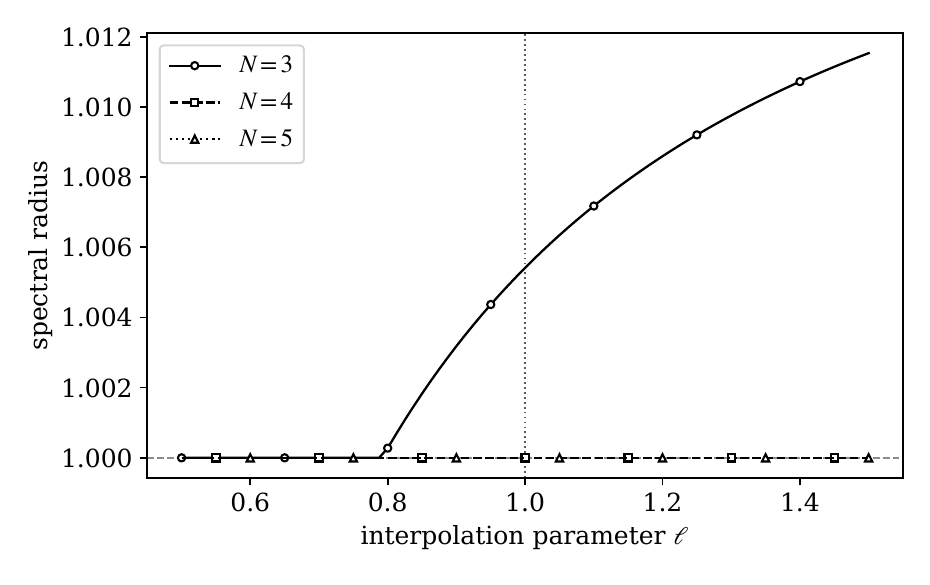}
\caption{Spectral radius as a function of $\ell$ for $\gamma=3/4$ on square domains; each curve joins 81 equally spaced samples.  For $N=3$, the vertical dotted line at $\ell=1$ marks the exact counterexample.  The $N=4$ and $N=5$ curves overlap at the plotting resolution.  The size dependence illustrates finite-domain effects.  The insufficiency of coefficient convexity as a stability condition is established separately by the exact counterexample.}
\label{fig:ell-curves}
\end{figure}

\begin{figure}[!htbp]
\centering
\includegraphics[width=.80\textwidth]{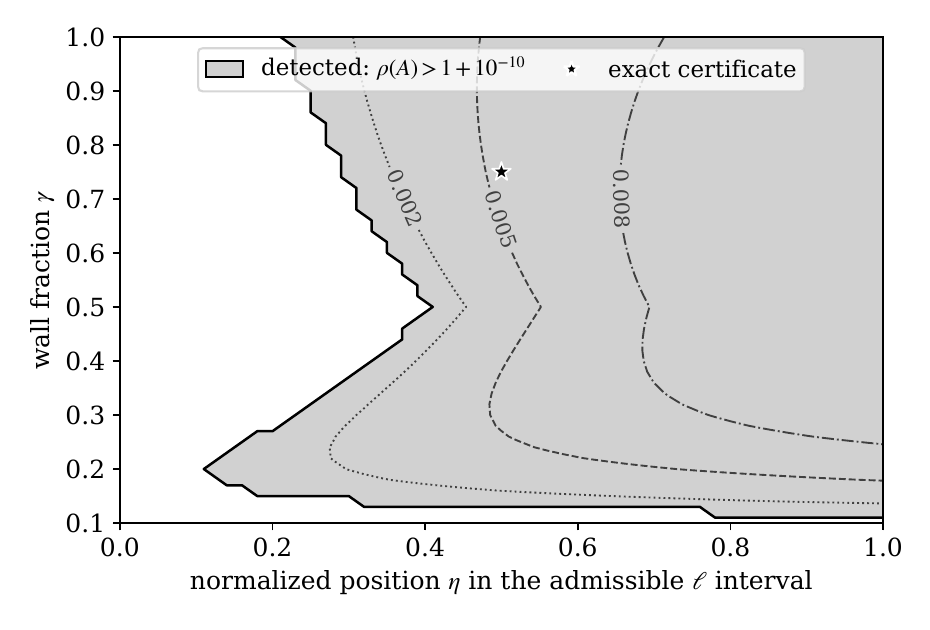}
\caption{Spectral radius on the $3\times3$ square.  For each $\gamma$, let $\ell_{\min}(\gamma)=\max\{0,2\gamma-1\}$ and $\ell=\ell_{\min}(\gamma)+\eta[2\gamma-\ell_{\min}(\gamma)]$.  On the $47\times51$ parameter grid, light gray marks samples with $\rho(A)>1+10^{-10}$, and the solid black curve traces the sampled boundary of this thresholded set.  The labeled dotted, dashed, and dash-dotted curves are interpolated contours of $\rho(A)-1$ at $0.002$, $0.005$, and $0.008$, respectively.  The star marks the exact counterexample $(\gamma,\ell,\eta)=(3/4,1,1/2)$.  White only means that no eigenvalue outside the unit disk was detected at this tolerance; it is not a proof of power stability.}
\label{fig:parameter-map}
\end{figure}

\begin{figure}[!htbp]
\centering
\includegraphics[width=.80\textwidth]{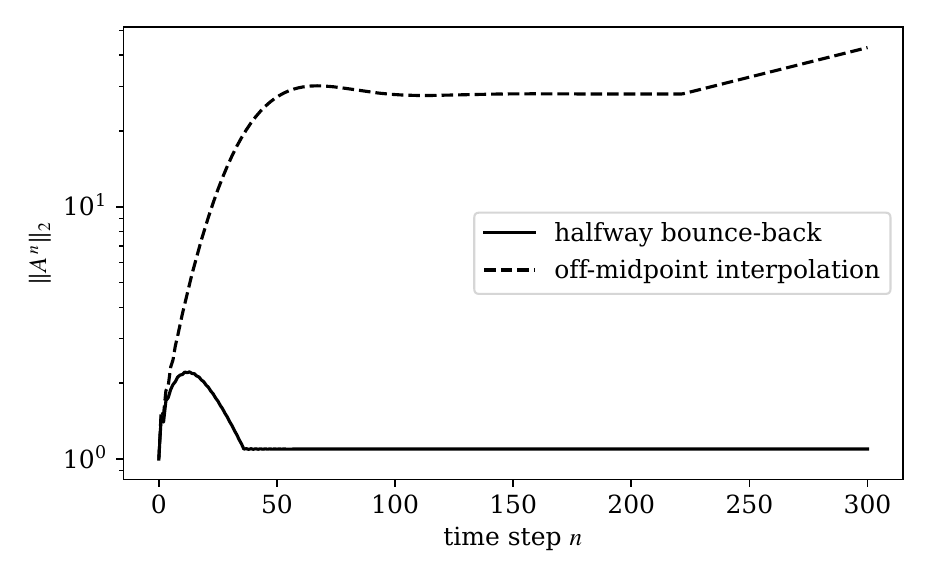}
\caption{Euclidean norms of powers on the $3\times3$ square.  The halfway operator is power bounded by Corollary~\ref{cor:boundaries}; for $0\leq n\leq300$, its computed power norm has maximum approximately $2.216$.  For the off-midpoint counterexample, the computed norm reaches approximately $42.8$ at step $300$.}
\label{fig:powers}
\end{figure}

\begin{figure}[!htbp]
\centering
\includegraphics[width=.80\textwidth]{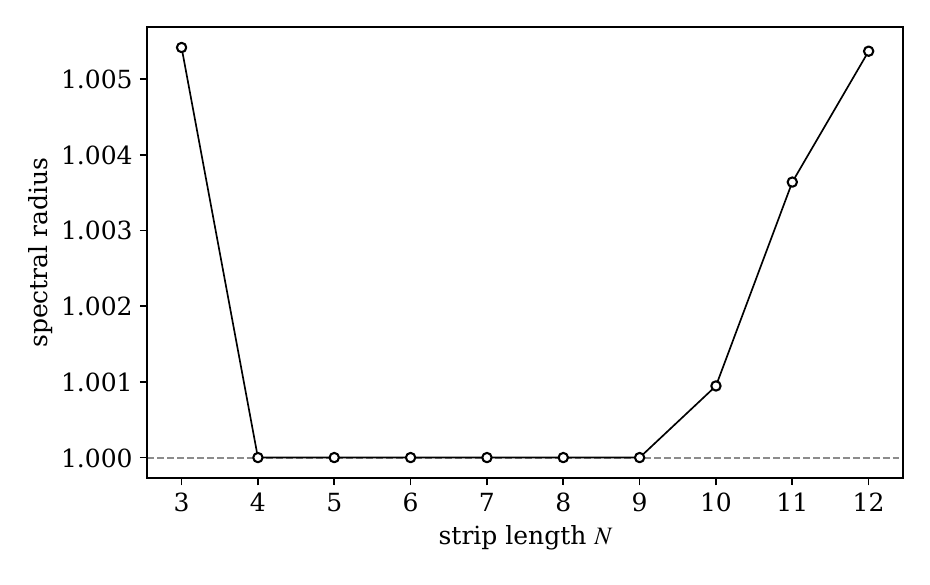}
\caption{Spectral radius for $N\times3$ strips with the exact counterexample parameters.  For $N=4,\ldots,9$, the computed values differ from one by at most $1.1\times10^{-14}$; this floating-point observation does not establish power stability.  The variation with $N$ records finite-domain dependence, and no resonance mechanism or mesh-asymptotic conclusion is inferred from this plot alone.}
\label{fig:strips}
\end{figure}

The accompanying data include the CSV files underlying Figures~\ref{fig:ell-curves}--\ref{fig:strips}.  Figure~\ref{fig:powers} also reports $\lVert A^n\rVert_2$, since spectral radius alone does not describe transient growth for a nonnormal matrix.

\end{document}